\documentclass[11pt]{article} 

\usepackage{amsmath,amsthm,amsfonts,amssymb,amscd,latexsym,eucal}
\usepackage[a4paper,top=3cm,bottom=4cm,left=3.5cm,right=3cm]{geometry}

\newtheorem{corollary}{Corollary}

\newtheorem{proposition}{Proposition}
\newtheorem{remark}{Remark}

\newtheorem{theorem}{Theorem}

\numberwithin{equation}{section}

\begin{document}

\title{Reversible part of a quantum dynamical system}

\author{Carlo Pandiscia}

\maketitle

\begin{abstract}
\indent 
In this work a quantum dynamical system $(\mathfrak M,\Phi, \varphi)$ is constituted by a von Neumann algebra $\mathfrak M$, by a unital Schwartz map $\Phi:\mathfrak{M\rightarrow M}$ and  by a $\Phi$-invariant normal faithful state $\varphi$ on $\mathfrak M$. 
We prove that the ergodic properties of a quantum dynamical system, are determined by its reversible part  $(\mathfrak{D}_\infty,\Phi_\infty, \varphi_\infty)$. It is constituted by a von Neumann sub-algebra $\mathfrak{D}_\infty$ of $\mathfrak M$ by an automorphism $\Phi_\infty$ and a normal state $\varphi_\infty$, the restrictions of $\Phi$ and $\varphi$ on $\mathfrak{D}_\infty$  respectively. 
Moreover, if $\mathfrak{D}_\infty$ is a trivial algebra then the  quantum dynamical system is ergodic. Furthermore we will give some  properties of the reversible part of quantum dynamical system, in particular, we will study its relations with the canonical decomposition of Nagy-Fojas of  linear contraction related to the quantum dynamical system.

\end{abstract}

\section{Preliminares and notations}
 
We consider a pair $(\mathfrak{M},\Phi) $ constituted  by a von Neumann algebra $\mathfrak{M}$ and a unital Schwartz map $\Phi:\mathfrak{M\rightarrow M}$  \textit{i.e.} a $\sigma$-continuous map with $\Phi(1)=1$  which satisfies the inequality: 
\begin{equation}
0 \leq \Phi(a^*)\Phi(a) \leq \Phi(a^*a)   \qquad  a\in\mathfrak{M} 
\label{schrwaz-ine}
\end{equation}
In this work the pair $(\mathfrak{M},\Phi) $  will  be called (discrete)  \textit{quantum process }and $\Phi$ the dynamics of the quantum process.
\\
A normal state $\varphi$ on $\mathfrak{M}$ is a stationary state for the quantum  process $(\mathfrak{M},\Phi)$ if $\varphi(\Phi(a))=\varphi(a)$ for all $a\in\mathfrak{M}$, while is of asymptotic equilibrium if $\Phi^n(a) \rightarrow \varphi(a) 1$ as $n \rightarrow \infty$ in $\sigma$-topology \textit{i.e.}
$$ \lim_{n\rightarrow +\infty} \omega(\Phi^n(a))=\omega (1) \varphi(a) \quad a\in \mathfrak M \quad \omega\in\mathfrak M_{*} $$
We denote with $\mathcal{B}(\mathcal{H})$ the C*-algebra of bounded linear operator on Hilbert space $\mathcal{H}$ and with $s$ and $\sigma$ respectively the ultrastrong operator topology and the ultraweakly operator topology on von Neumann algebra $\mathfrak{M}$ while with $\mathfrak M_{*}$ its predual. Furthermore, a normal map or a normal state are $\sigma$-continuous maps (see ref. \cite{bratteli79}).
\\
We define the \textit{multiplicative domain} $\mathfrak{D}_{\Phi}$  of a Schwartz map (see definition 2.1.4 and proposition 2.1.6 of \cite{stormer})  as follows:
\[
\mathfrak{D}_{\Phi}= \{ a\in\mathfrak{M}: \Phi(a^*a)=\Phi(a^*)\Phi(a) \ \ and \ \  \Phi(aa^*)=\Phi(a)\Phi(a^*) \}
\]
We recall that an element $a\in\mathfrak{D}_{\Phi}$ if and only if $\Phi(ax)=\Phi(a)\Phi(x)$ and $\Phi(xa)=\Phi(x)\Phi(a)$ for all $x\in\mathfrak{M}$.
It follows that $\mathfrak{D}_{\Phi}$ is a von Neumann algebra, since it is a unital *-algebra closed in the $\sigma$-topology.
\smallskip
\\
A consequence of the Schwartz's inequality is the following remark:
\begin{remark}
If $\Phi:\mathfrak{M}\rightarrow\mathfrak{M}$ is a unital Schwartz map which admits an inverse $\Phi^{-1}:\mathfrak{M}\rightarrow\mathfrak{M}$ (\textsl{i.e.} a unital Schwartz map such that $\Phi(\Phi^{-1}(a))=\Phi^{-1}(\Phi(a))=a $ for all $a\in\mathfrak{M}$), then $\Phi$ is an automorphism.
\end{remark}
If $\mathfrak{D}_{\infty}^+$ is the following von Neumann algebras:
\begin{equation}
\label{dpos}
\mathfrak{D}_{\infty}^+=\bigcap_{n\in\mathbb{N}}\mathfrak{D}_{\Phi^n}
\end{equation}
then we have that $\Phi(\mathfrak{D}_{\infty}^+)\subset\mathfrak{D}_{\infty}^+$ and $\Phi$ restricted to $\mathfrak{D}_{\infty}^+$ is a *-homomorphism, but it is not surjective map.
\\
Moreover we have:
$$\mathfrak{D}_{\infty}^+= \left\{a\in\mathfrak D_\Phi: \Phi^n(a)\in\mathfrak D_\Phi \ \textsl{for all} \ n\in\mathbb N \right\} $$
We define the \textit{multiplicative core} of $\Phi$ (see ref. \cite{stormer07}):
$$ \mathcal C_\Phi= \bigcap_{n\in\mathbb{N}} \Phi^n (\mathfrak{D}_{\infty}^+)\subset \mathfrak{D}_{\infty}^+ $$
We have  $\Phi(\mathcal C_\Phi )\subset \mathcal C_\Phi$.
\\
Indeed $\Phi^{n+1} (\mathfrak{D}_{\infty}^+)\subset\Phi^n (\mathfrak{D}_{\infty}^+)$ for all $n\geq0$ and  
$$ \Phi(\bigcap_{n\in\mathbb{N}}\Phi^{n}(\mathfrak{D}_{\infty}^+)) \subset  \bigcap_{n\in\mathbb{N}}\Phi(\Phi^{n}(\mathfrak{D}_{\infty}^+))=\bigcap_{n\in\mathbb{N}}\Phi^{n+1}(\mathfrak{D}_{\infty}^+)=\bigcap_{n\in\mathbb{N}} \Phi^{n}(\mathfrak{D}_{\infty}^+) $$
It is clear that the restriction of $\Phi$ to multiplicative core $\mathcal C_\Phi$ is a *-homomorphism and if $\Phi$ is an injective map on $ \mathfrak{D}_{\infty}^+$, then we have $\Phi( \mathcal C_\Phi)=  \mathcal C_\Phi$, so the restriction of $\Phi$ to the multiplicative core is *-automorphism.
\\
Since $\Phi$ is a normal map and its restriction to $ \mathfrak{D}_{\infty}^+$ is a *-homomorphism, we have that the set $\Phi^n(\mathfrak{D}_{\infty}^+)$ is a von Neumann algebra (see e.g. \cite{bratteli79}), therefore $\mathcal C_\phi$ is a von Neumann algebra. 
\\
\indent
Let $\varphi$ be a stationary state for the quantum processes $(\mathfrak{M},\Phi) $ and $(\mathcal{H}_{\varphi},\pi_{\varphi},\Omega_{\varphi})$ its GNS representation. It is well know (see e.g. \cite{NSZ} ) that there is a unique linear contraction $U_{\Phi,\varphi}$ of $\mathfrak{B}(\mathcal{H}_{\varphi})$ such that, for any $a\in\mathfrak{A}$, we have
\begin{equation}
U_{\Phi,\varphi}\pi_{\varphi}(a)\Omega_{\varphi}=\pi_{\varphi}(\Phi(a))\Omega_{\varphi} 
\label{contrazione 1}
\end{equation}
Furthermore if $\varphi$ is a faithful state then there is a unital Schwartz map $\Phi_{\bullet}:\pi_{\varphi}(\mathfrak{M})\rightarrow\pi_{\varphi}(\mathfrak{M})$ such that
\begin{equation}
\label{phibullet}
\Phi_{\bullet}(A)\Omega_{\varphi}=U_{\Phi,\varphi}A\Omega_{\varphi} \qquad  A\in\pi_{\varphi}(\mathfrak{M}) 
\end{equation}
It is simple to prove the following statements on multiplicative domains of Schwartz maps:
\begin{proposition}
\label{prop0}
Let $(\mathfrak{M},\Phi) $ be quantum processes and $\varphi$ its faithful stationary state, we have:
\item{a]} For each $d\in\mathfrak{D}_{\Phi}$ \  results  \  $U_{\Phi,\varphi}\pi_{\varphi}(d)=\pi_{\varphi}(\Phi(d))U_{\Phi,\varphi}$
\item{b]} If $U_{\Phi,\varphi}\pi_{\varphi}(a)=\pi_{\varphi}(\Phi(a))U_{\Phi,\varphi}$ \   then \  $\Phi(ax)=\Phi(a)\Phi(x)$ for all $x\in\mathfrak{M}$
\item{c]} $U_{\Phi,\varphi}^*U_{\Phi,\varphi}\in\pi_{\varphi}(\mathfrak{D}_{\Phi})'$  \  while \  $U_{\Phi,\varphi}U_{\Phi,\varphi}^*\in\pi_{\varphi}(\Phi(\mathfrak{D}_{\Phi}))'$
\item {d] }  $d\in\mathfrak{D}_{\Phi}$  \  if, and only if \  $||U_{\Phi,\varphi}\pi_{\varphi}(d)\Omega_{\varphi}||=||\pi_{\varphi}(d)\Omega_{\varphi}||$ \ and 
\ $||U_{\Phi,\varphi}\pi_{\varphi}(d^*)\Omega_{\varphi}||=||\pi_{\varphi}(d^*)\Omega_{\varphi}||$
\end{proposition}
\begin{proof}
It is  straightforward
\end{proof}
We observe that if $\Phi$ is a *-homomorphism, then the contraction $U_{\Phi,\varphi}$ is an isometry on $\mathcal{H}_{\varphi}$.
\\
Another trivial consequence  of Schwartz's inequality and of the existence of a faithful stationary state for quantum process $(\mathfrak{M},\Phi) $, are the following relations:
\begin{equation}
\label{inclusion}
 \cdots\mathfrak{D}_{\Phi^n}\subset\mathfrak{D}_{\Phi^{n-1}}\subset\cdots\mathfrak{D}_{\Phi^2}\subset\mathfrak{D}_{\Phi}\subset\mathfrak{M}
\end{equation}
for all natural numbers $n\in\mathbb{N}$.
\\
Furthermore, since  $\Phi$ is a injective map on $\mathfrak{D}_{\infty}^+$,  its restricted to $\mathcal C_\Phi$ is a *-automorphism.
\\
In fact, if $a\in\mathfrak{D}_{\infty}^+$ with $\Phi(a)=0$, then we obtain
$$\varphi(\Phi(a^*) \Phi(a))=\varphi(\Phi(a^* a))=\varphi(a^*a)=0$$ 
Let $(\mathfrak{M},\Phi)$ be a quantum processes and $\varphi$ its stationary state, we recall that the dynamics $\Phi$ admits a $\varphi$-adjoint if there is a normal  unital Schwartz map $\Phi^{\sharp}:\mathfrak{M}\rightarrow\mathfrak{M}$ such that
\[
\varphi(b\Phi(a))=\varphi(\Phi^{\natural}(b)a) \qquad a,b\in\mathfrak{M} 
\]
We have the following conditions for the existence of a $\varphi$-adjointness of dynamics of quantum process (see proposition 3.3 in \cite{NSZ}):
\begin{proposition}
\label{prop-NSZ}
Let $(\mathfrak{M},\Phi)$ be a quantum process and $\varphi$ its faithful stationary state. If $(\Delta_{\varphi},J_{\varphi})$ denote the modular operators associated with pair $(\pi_{\varphi}(\mathfrak{M}),\Omega_{\varphi})$, then the following conditions are equivalent:
\begin{itemize}
\item[1 - ] $\Phi$ commutes with the modular automorphism group  $\left\{ \sigma_{t}^{\varphi} \right\}_{t\in\mathbb{R}} $ i.e.
\[
\sigma_{t}^{\varphi}\circ\Phi_{\bullet}=\Phi_{\bullet}\circ\sigma_{t}^{\varphi} \qquad  t\in\mathbb{R} 
\]
\item[2 - ] $U_{\Phi,\varphi}$ commutes with modular operators:
\[
U_{\Phi,\varphi}\Delta_{\varphi}^{it}=\Delta_{\varphi}^{it}U_{\Phi,\varphi}  \qquad t\in\mathbb{R} 
\]
and
\[
U_{\Phi,\varphi}J_{\varphi}=J_{\varphi}U_{\Phi,\varphi}  
\]
\item[3 - ] $\Phi$ admits $\varphi$-adjoint $\Phi^\sharp$.
\end{itemize}
\end{proposition}
A triple $(\mathfrak{M},\Phi, \varphi)$ constituted  by quantum processes $(\mathfrak{M},\Phi)$, by its normal faithful stationary state $\varphi$ and with dynamics  $\Phi$ which admits a $\varphi$-adjoint $\Phi^\sharp$, will be called a quantum dynamical system.
\section{Decomposition theorem}
We consider a von Neumann algebra $\mathfrak M$ and its faithful normal state $\varphi$ and set with $(\mathcal{H}_{\varphi},\pi_{\varphi},\Omega_{\varphi})$ the GNS representation of $\varphi$ and with $\left\{\sigma_{t}^{\varphi}\right\}_{t\in \mathbb R}$ its modular automorphism group. 
\\
Let $\mathfrak{R}$ be a von Neumann subalgebra of  $\mathfrak{M}$,  we recall (see ref. \cite{kummerer85}) that the $\varphi$-orthogonal of $\mathfrak{R}$ is the set:
\begin{equation}
\mathfrak{R}^{\perp_{\varphi}}=\{ a\in\mathfrak{M}:\varphi(a^*x)=0 \ \ \ \text{for all}\ \ x\in\mathfrak{R} \}
\end{equation} 
Furthermore, it is simple to prove that $\mathfrak{R}^{\perp_{\varphi}}$ is a closed linear space  in the $\sigma$-topology with $\mathfrak{R}^{\perp_{\varphi}}\cap\mathfrak{R}$=\{0\}.
\\
We observe that $\mathfrak{R}^{\perp_{\varphi}}\subset \ker\varphi$ and if $\mathfrak{R}=\mathbb{C}I$ then $\mathfrak{R}^{\perp_{\varphi}}= \ker\varphi$, where $\ker\varphi=\{a\in\mathfrak{M}: \ \varphi(a)=0 \}$.
\\
Moreover if $ y \in \mathfrak R $ and $d_\bot \in \mathfrak{R}^{\perp_{\varphi}}$ then $y d_\bot \in \mathfrak{R}^{\perp_{\varphi}}$ since 
$$ \varphi( (y d_\bot)^* x)=\varphi (d_\bot^* y ^* x)=0 \qquad x\in\mathfrak R $$
\begin{theorem}
\label{teo-dec}
The von Neumann algebra $\mathfrak R$ is invariant under modular automorphism group $\sigma_{t}^{\varphi}$ if and only if both these conditions are fulfilled:  
\begin{itemize}
\item[a - ] the set $\mathfrak{R}^{\perp_{\varphi}}$  is closed under the involution operation;
\item[b - ] for any $a\in\mathfrak M$ there is a unique $a_\| \in \mathfrak R$ and $a_\bot \in \mathfrak{R}^{\perp_{\varphi}}$  such that $a= a_\| + a_\bot $.
\\In other words we have the following algebraic decomposition
\begin{equation}
\label{R-dec}
 \mathfrak{M}= \mathfrak{R} \oplus \mathfrak{R}^{\perp_{\varphi}}
\end{equation}
\end{itemize}  
\end{theorem}
\begin{proof}
From Takesaki \cite{takesaki72} we have $\sigma_{t}^{\varphi} (\pi_\varphi(\mathfrak R)) \subset \pi_\varphi(\mathfrak {R})$ for all $t\in\mathbb{R}$ if, and only if there exist a  normal conditional expectation $\mathcal E:\mathfrak M \rightarrow \mathfrak R $ such that $\varphi \circ \mathcal E = \varphi$.
\\
Let $\mathfrak R$ be  invariant under modular automorphism group $\sigma_{t}^{\varphi}$, it is simple to prove that  
\begin{equation}
\mathfrak{R}^{\perp_{\varphi}}=\left\{a\in\mathfrak M : \mathcal E (a)=0 \right\}
\end{equation}
hence $\mathfrak{R}^{\perp_{\varphi}}$ is closed under the involution operation.
\\
For any $a\in\mathfrak M$ we set  $a_\bot=a- \mathcal E (a)$  and
$$\varphi(a_\bot^* x)= \varphi((a^* - \mathcal E (a^*))x)=\varphi(a^* x) - \varphi( \mathcal E (a^*)x)= \varphi(a^* x) - \varphi( \mathcal E (a^*x))=0$$
for all $x\in \mathfrak R$ hence $a_\bot\in \mathfrak{R}^{\perp_{\varphi}}$.
\\
So for any $a\in \mathfrak M$ there exist a unique $a_\| \in \mathfrak R$ and $a_\bot \in \mathfrak{R}^{\perp_{\varphi}}$ such that $a= a_\| + a_\bot $  where we have set $a_\| =\mathcal E(a)$. 
\\
The uniqueness follows because if $a=0$ then $a_\| =a_\bot=0$. 
\\
Indeed we have
$$ \varphi(a^* a) = \varphi(a_\|^* a_\|) + \varphi(a_\bot^*a_\bot)=0$$ 
since $a_\|^* a_\bot, $ and $ a_\bot^*a_\|$ belong to $\mathfrak{R}^{\perp_{\varphi}}$ and $\varphi$ is a faithful state.
\\
For the vice-versa, if the set $\mathfrak{R}^{\perp_{\varphi}}$ is closed under the involution operation and $ \mathfrak{M}= \mathfrak{R} \oplus \mathfrak{R}^{\perp_{\varphi}}$ then for any $a\in \mathfrak M$ there is a unique $a_\| \in \mathfrak R$ and $a_\bot \in \mathfrak{R}^{\perp_{\varphi}}$ such that $a= a_\| + a_\bot $ . 
\\
The map $ a \in \mathfrak M \rightarrow a_\| \in \mathfrak R $ is a projection of norm one ( i.e. it is satisfies $(1)_\|=1$ and $ ( (a)_\| )_\| =a_\|$ for all $a\in \mathfrak M $ ), for Tomiyama \cite{tomiyama57} it is a normal conditional expectation (see \cite{li} for a modern review) and $\varphi(a)=\varphi(a_\|)$ for all $a \in \mathfrak M$.
\end{proof}
We observe that if $\mathfrak R^{\perp_{\varphi}}$ is a *-algebra (without unit) then $\mathfrak R^{\perp_{\varphi}}= \left\{0 \right\}$ since $\varphi(a_\bot^* a_\bot)=0$ for all $a_\bot\in\mathfrak R^{\perp_{\varphi}}$ and $\varphi$ is a faithful state.
\\
 Moreover, if $p$ is a orthogonal projector of $\mathfrak M$ then $p\notin \mathfrak R^{\perp_{\varphi}}$.
\\
We have the following remark:
\\
If $a\in\mathfrak{M}$ with $a=a_\| +a_\bot$ where $a_\|\in\mathfrak{R}$ and $a_\bot\in\mathfrak{R}^{\perp_{\varphi}}$, then 
\begin{equation}
\label{hilbertdeco1}
||\pi_{\varphi}(a)\Omega_{\varphi}||^2=||\pi_{\varphi}(a_\|)\Omega_{\varphi}||^2+||\pi_{\varphi}(a_\bot)\Omega_{\varphi}||^2
\end{equation}
\begin{proposition}
\label{hilbertdeco2}
Let $\mathfrak R$ be a von Neumann algebra invariant under modular automorphism group $\sigma_{t}^{\varphi}$ . If $\mathcal{H}_o$ and $\mathcal{K}_o$ are the closure of the linear space $\pi_{\varphi}(\mathfrak{R})\Omega_{\varphi}$ and of $\pi_{\varphi}(\mathfrak{R}^{\perp_{\varphi}})\Omega_{\varphi}$ respectively, then
\[
\mathcal{H}_{\varphi}=\mathcal{H}_o \oplus \mathcal{K}_o
\]
Moreover the orthogonal projection $P_o$ on Hilbert space $\mathcal{H}_o$ belongs to $\pi_{\varphi}(\mathfrak{R})'$.
\end{proposition}
\begin{proof}
We have that $ \mathcal K_o\subset \mathcal H_o^\bot$ since for any $r_\bot\in\mathfrak{R}^{\perp_\varphi}$ and $\psi_o\in\mathcal{H}_o$ we obtain:
\begin{equation*}
\left\langle\pi_\varphi(r_\bot)\Omega_\varphi, \psi_o \right\rangle = \underset{\alpha\rightarrow\infty}{\lim}\left\langle\pi_\varphi(r_\bot)\Omega_\varphi, \pi_\varphi(r_\alpha )\Omega_\varphi \right\rangle= \underset{\alpha\rightarrow\infty}{\lim} \varphi(r_\bot^*r_\alpha)=0
\end{equation*}
where $\psi_o =\underset{\alpha\rightarrow\infty}{\lim} \pi_\varphi(r_\alpha )\Omega_\varphi$ with $\{r_\alpha \}_\alpha$ net belongs to $\mathfrak{R}$.
\\
Let $\psi\in\mathcal{H}_\varphi$ we can write 
\[
\psi=\underset{\alpha\rightarrow\infty}{\lim}\pi_{\varphi}(m_\alpha)\Omega_\varphi =\underset{\alpha\rightarrow\infty}{\lim}(\pi_{\varphi}(r_\alpha)\Omega_\varphi+\pi_{\varphi}((r_{\alpha \bot})\Omega_\varphi)
\]
where $m_\alpha = r_\alpha +  r_{\alpha \bot}$ for each $\alpha$.
\\
The net $ \{ \pi_{\varphi}(r_\alpha)\Omega_{\varphi} \}$ has limit, since by the relation (\ref{hilbertdeco1}) for each $\epsilon\geq0$ there is a index $\nu$ such  that for $\alpha\geq\nu$ and $\beta\geq\nu$ we have the Cauchy relation:  
\[
|| \pi_\varphi(r_\alpha)\Omega_\varphi-\pi_\varphi(r_\beta)\Omega_\varphi || \leq || \pi_\varphi(m_\alpha)\Omega_\varphi -\pi_\varphi(m_\beta)\Omega_\varphi ||\leq \epsilon
\]
It follows that there are $\psi_\|\in\mathcal{H}_o$ and $\psi_\bot\in \mathcal{K}_o$ such that
\[
\psi=\underset{\alpha\rightarrow\infty}{\lim}\pi_{\varphi}(r_\alpha)\Omega_\varphi + \underset{\alpha\rightarrow\infty}{\lim}\pi_{\varphi}((r_{\alpha \bot})\Omega_\varphi=
\psi_\| + \psi_\bot\in \mathcal{H}_o \oplus \mathcal{K}_o
\]
It is simple to prove that  $\pi_\varphi(\mathfrak{R}) \mathcal{H}_o\subset \mathcal{H}_o $ therefore   $P_o\in \pi_\varphi(\mathfrak{R})'$.
\end{proof}
We have the following proposition:
\begin{proposition}
Let $( \mathfrak{M},\Phi)$ be a quantum process and $\varphi$ a normal faithful state on $\mathfrak M$. For any natural number $n \in \mathbb N$ we obtain:
\begin{equation}
 \mathfrak M = \mathfrak D_{\Phi^n} \oplus \mathfrak D_{\Phi^n}^{\perp_\varphi}
\end{equation}
and
\begin{equation}
\label{teo-dec1}
\mathfrak M = \mathfrak D_\infty^+ \oplus \mathfrak D_\infty^{+ \perp_\varphi}
\end{equation}
Furthermore, if $\varphi$ is a stationary state for  $\Phi$, then
\begin{equation}
\label{teo-dec1b}
\mathfrak M = \mathcal C_\Phi \oplus \mathcal C_\Phi ^{\perp_\varphi}
\end{equation}
and the restriction of $\Phi$ to $\mathcal C_\Phi$ is a *-automorphism with $\Phi(\mathcal C_\Phi ^{\perp_\varphi})\subset \mathcal C_\Phi ^{\perp_\varphi}$.
\end{proposition}
\begin{proof}
for any $d\in\mathfrak D_{\Phi^n}$ and natural number $n$ we have: 
\begin{eqnarray*}
\Phi_{\bullet}^ n(\sigma^{\varphi}_{t}(\pi_{\varphi}(d)^*)\sigma^{\varphi}_{t}(\pi_{\varphi}(d))) =\Phi_{\bullet}^n(\sigma^{\varphi}_{t}(\pi_{\varphi}(d)^*))\Phi_{\bullet}^n(\sigma^{\varphi}_{t}(\pi_{\varphi}(d))).
\end{eqnarray*}
since $\Phi$ commutes with our modular automorphism group $\sigma^{\varphi}_{t}$. It follows that 
$\sigma^{\varphi}_{t}(\pi_\varphi(\mathfrak D_{\Phi^n}) )$ is included in $\pi_\varphi(\mathfrak D_{\Phi^n})$ for all $n \in \mathbb N$ and $t \in \mathbb R$.
\\
Let $b\in\mathcal C_\Phi$, we have that $\sigma^{\varphi}_{t}(\pi_{\varphi}(b))\in\pi_\varphi(\mathcal C_\Phi)$ for all real number $t$.
\\
In fact for each natural number $n$ there exist a $x_n\in \mathfrak D_\infty^{+}$ such that $b=\Phi^n(x_n)$. We can write that 
$$\sigma^{\varphi}_{t}(\pi_\varphi(b)) = \sigma^{\varphi}_{t}(\pi_\varphi(\Phi^n(x_n))=\Phi^n_\bullet(\sigma^{\varphi}_{t}(\pi_\varphi(x_n)) $$ 
and by above relation $\sigma^{\varphi}_{t}(\pi_\varphi(x_n))\in \pi_\varphi(\mathfrak D_\infty^+)$ for all natural number $n$. It follows that 
$\sigma^{\varphi}_{t}(\pi_\varphi(b))\in \pi_\varphi (\Phi^n(\mathfrak D_\infty^+)$ for all natural number $n$.
\\
Let $y\in C_\Phi ^{\perp_\varphi}$, since $\Phi(\mathcal C_\Phi)= \mathcal C_\Phi$ we have for any $c\in\mathcal C_\Phi$ that 
$$ \varphi(\Phi(y) c )=\varphi(\Phi(y) \Phi(c_o))=\varphi(y c_o)=0$$
where $c=\Phi(c_o)$ with $c_o\in \mathcal C_\Phi\subset \mathfrak D_\infty^+$.
\end{proof}
We consider a quantum  dynamical system $( \mathfrak{M},\Phi,\varphi)$ with $\varphi$-adjoint $\Phi^\sharp$.  We set with $\mathfrak{D}_{\infty}$ (or with $\mathfrak{D}_{\infty}(\Phi)$ when we have to highlight the map $\Phi$), the following von Neumann Algebra
\begin{equation}
\mathfrak{D}_{\infty} = \bigcap \limits_{k\in\mathbb{Z}} \mathfrak{D}_{\Phi_{k}}
\end{equation}
where for each $k$ integer we denote
\[
\Phi_{k}=\left\{
\begin{array}
[c]{cc}
\Phi^{k} & k\geq0
\\
 \ \ \Phi^{\sharp\left\vert k\right\vert } & k<0
\end{array}
\right.
\]
while with $\mathfrak{D}_{\Phi_{k}}$ we have set the von Neumann algebra of the multiplicative domains of the dynamics $\Phi_{k}$.
\\
Following \cite{robinson82}, for each $a,b\in\mathfrak{M}$ and integers $k$ we define:
\begin{equation}
\label{sesqui}
S_{k}(a,b)=\Phi_{k}(a^*b)-\Phi_{k}(a^*)\Phi_{k}(b)\in\mathfrak{M}
\end{equation}
and we have these simple relations:
\begin{itemize}
\item[ \textit{a} - ] $S_{k}(a,a)\geq0$ for all $a\in\mathfrak{M}$ and integers $k$;
\item[ \textit{b} - ] $S_{k}(a,b)^*=S_{k}(b,a)$ for all $a,b\in\mathfrak{M}$ and integers $k$; 
\item[ \textit{c} - ] $d\in\mathfrak{D}_{\infty}$ if, and only if $S_{k}(d,d)=S_{k}(d^*,d^*)=0$ \qquad for all integers $k$;
\item[ \textit{d} - ] $d\in\mathfrak{D}_{\infty}$ if, and only if $\varphi(S_{k}(d,d))=\varphi(S_{k}(d^*,d^*))=0$ \quad for all integers $k$; 
\item[ \textit{e} - ] The map $a,b\in\mathfrak{M}\rightarrow\varphi(S_{k}(a,b))$ for all integers $k$, is a sesquilinear form, hence
\[
|\varphi(S_{k}(a,b))|^2\leq\varphi(S_{k}(a,a)) \varphi(S_{k}(b,b)) \qquad \ a,b\in\mathfrak{M} 
\]
\end{itemize}
We observe that $\Phi(\mathfrak{D}_{\infty})\subset\mathfrak{D}_{\infty}$ and $\Phi^{\sharp}(\mathfrak{D}_{\infty})\subset\mathfrak{D}_{\infty}$. Indeed for each element $d\in\mathfrak{D}_{\infty}$ and integer $k$ we have
\newline
\[
\varphi(S_{k}(\Phi(d),\Phi(d))=\varphi(S_{k+1}(d,d)=0
\]
and
\[
\varphi(S_{k}(\Phi^{\sharp}(d),\Phi^{\sharp}(d))=\varphi(S_{k-1}(d,d)=0
\]
Furthermore $d^*\in\mathfrak{D}_{\infty}$ thus we obtain also 
\[
\varphi(S_{k}(\Phi(d)^*,\Phi(d)^*)=\varphi(S_{k}(\Phi^{\sharp}(d)^*,\Phi^{\sharp}(d)^*)=0
\]
It follows that restriction of the map $\Phi$ at von Neumann algebra $\mathfrak{D}_{\infty}$ it is a *-automorphism where $\Phi(\Phi^{\sharp}(d))=\Phi^{\sharp}(\Phi(d))=d$ for all $d\in\mathfrak{D}_{\infty}$. 
\\

We summarize the results obtained in following statement:
\begin{proposition}
\label{automorphism}
Let $( \mathfrak{M},\Phi,\varphi)$ be a quantum dynamical system. The map $\Phi_{\infty}:\mathfrak{D}_{\infty}\rightarrow\mathfrak{D}_{\infty}$ where $\Phi_{\infty}(d)=\Phi(d)$ for all $d\in\mathfrak{D}_{\infty}$, is a *-automorphism of von Neumann algebra. 
\newline
Furthermore if there is a von Neumann subalgebra $\mathfrak B$ of $\mathfrak{M}$ such that the restriction of $\Phi$ to $\mathfrak B$ is a *-automorphism, then we obtain $\mathfrak B\subset\mathfrak{D}_{\infty}$.
\newline
We have a (maximal) reversible quantum dynamical systems  $(\mathfrak{D}_{\infty},\Phi_{\infty},\varphi_{\infty})$ where  the normal state $\varphi_{\infty}$ and the $\varphi_{\infty}$-adjoint $\Phi_{\infty}^{\sharp}$, are respectively the restriction of $\varphi$ and $\Phi^{\sharp}$ to the von Neumann algebra $\mathfrak{D}_{\infty}$. 
\end{proposition}
\begin{proof}
We prove that if the restriction of $\Phi$ to $\mathfrak B$ is an automorphism, then $\mathfrak B \subset \mathfrak D_\infty$.
\\
In fact we have that $\mathfrak B \subset \mathfrak D_{\Phi^n}$ for all natural number $n$ and if $\Psi:\mathfrak B \rightarrow \mathfrak B$ is the map such that $\Psi(\Phi(b))=\Phi(\Psi(b))=b$ for all $b \in\mathfrak B$, then $\Psi(b)=\Phi^\sharp(b)$, since   
\[
\varphi(a \Psi(b) )= \varphi(\Phi(a \Psi(b))) =\varphi(\Phi(a) \Phi(\Psi(b))=\varphi (\Phi(a) b))=\varphi( a \Phi^\sharp(b) )
\]
for all $a \in \mathfrak M$. It follows that $\mathfrak B$ is also $\Phi^\sharp$-invariant, hence $\mathfrak B \subset \mathfrak D_{ \Phi^{n \sharp}}$ for all natural number $n$.
\end{proof}
It is clear that $\mathfrak D_\infty$ is $\Phi_k$-invariant for all integers $k$ and is invariant under automorphism group $\sigma_t^\varphi$ and by previous decomposition theorem we can say that (see \cite{hel-Blac} theorem 6):
\begin{proposition}
\label{teo-dec2}
If $( \mathfrak{M},\Phi,\varphi)$ is a quantum dynamical system, then there is a conditional expectation $\mathcal E_\infty:\mathfrak M \rightarrow \mathfrak D_\infty$such that
\begin{itemize}
\item[a -] $\varphi \circ \mathcal E_\infty=\varphi$;
\item[b -] $\mathfrak D_\infty^{\perp_\varphi}=\ker \mathcal E_\infty$;
\item[c -] $\mathfrak M = \mathfrak D_\infty \oplus \mathfrak D_\infty^{\perp_\varphi}$;
\item[d -] $\Phi_k (\mathfrak D_\infty^{\perp_\varphi}) \subset \mathfrak D_\infty^{\perp_\varphi}$ for all integers $k$;
\item[e -] $\mathcal E_\infty(\Phi_k (a))=\Phi_k (\mathcal E_\infty(a))$ for all $a\in\mathfrak M$ and integer $k$;
\item[f -] $\mathcal H_\varphi=\mathcal H_\infty \oplus \mathcal K_\infty$ where $\mathcal H_\infty$ and $\mathcal K_\infty$ denotes the linear closure of $\pi_\varphi (\mathfrak D_\infty)\Omega_\varphi$ and of $\pi_{\varphi}(\mathfrak D_\infty^{\perp_{\varphi}})\Omega_{\varphi}$ respectively.
\end{itemize}
 \end{proposition}
\begin{proof}
The statements $(a)$, $(b)$ and $(c)$  are simple consequence of theorem \ref{teo-dec}.
\\
For the statement $(d)$, if $ d_\bot \in \mathfrak D_\infty^{\perp_\varphi}$ then for any integer $k$ and $x \in \mathfrak D_\infty$, we have:
\[
\varphi(\Phi_k(d_\bot)^*x)=\varphi(d_\bot^* \Phi_{-k}(x))=0
\]
since $\Phi_{-k}(x))\in \mathfrak D_\infty$.
\\
For the statement $(e)$, for any $a,b \in \mathfrak M$ we obtain
\begin{eqnarray*}
\varphi(b \mathcal E_\infty(\Phi_k(a))) &=& \varphi((b_\| + b_\bot)\mathcal E_\infty(\Phi_k(a)))=\varphi(b_\|\mathcal E_\infty(\Phi_k(a))=
\varphi(\mathcal E_\infty(b\| \Phi_k(a)))=
\\
&=& \varphi( b\| \Phi_k(a)))= \varphi(\Phi_{-k}(b_\|) a)=\varphi(\mathcal E_\infty(\Phi_{-k}(b_\|)a))=
\\
&=& \varphi(\mathcal E_\infty(\Phi_{-k} (b_\|)a))=
\varphi(\Phi_{-k}(b_\|) \mathcal E_\infty(a))= \varphi( b_\| \Phi_k(\mathcal E_\infty (a))=
\\
&=&  \varphi( (b_\| +b_\bot ) \Phi_k(\mathcal E_\infty (a))= \varphi( b \Phi_k(\mathcal E_\infty (a)) 
\end{eqnarray*}
where we have write $b=b_\| + b_\bot$ with $b_\|=\mathcal E_\infty(b)$.
\end{proof}
The quantum dynamical system $( \mathfrak D_\infty,\Phi_\infty,\varphi_\infty)$ is called \textit{the reversible part }of the quantum dynamical system $( \mathfrak{M},\Phi,\varphi)$.
\\
Furthermore a quantum dynamical system is called \textit{completely irreversible} if $\mathfrak D_\infty=\mathbb C 1$.
\\
In this case for all  $a\in\mathfrak M$ we obtain $a= \varphi(a) 1 + a_\bot $ and we can write
$$ \mathfrak M = \mathbb C 1 \oplus \ker \varphi $$
Let $( \mathfrak{M},\Phi,\varphi)$ be a completely irreversible quantum dynamical system, if the von Neumann algebra $\mathfrak M$ is not trivial then there is least a not trivial projector $P\in\mathfrak M$, such that 
\begin{equation}
\label{purestate}
 \varphi (P) - \varphi (P)^2 >0 
\end{equation}
In fact,  we can write  $P=\varphi(P) 1 + P_\bot$ where $P_\bot \in \mathfrak D_\infty^{\perp_\varphi}$ and $\varphi(P_\bot ^2)=  \varphi (P) - \varphi (P)^2$ because $P_\bot ^2 + 2 \varphi(P) P_\bot + \varphi(P)^2 1= \varphi(P) 1 + P_\bot$- Therefore if  $\varphi (P)=\varphi (P)^2$, then $P_\bot = 0 $.
\\
In section 4 we will find the conditions when $\mathfrak D_\infty= \mathbb C 1$ (see also \cite{carbone2} section 2  for the case $\mathfrak D_\infty^+= \mathbb C 1$).
\\
We observe that if $\mathcal A(\mathcal P)$ is the von Neumann algebra generated by the set of all orthogonal projections $p\in\mathfrak M$ such that $\Phi_k(p)=\Phi_k(p)^2$ for all integers $k$, then $\mathfrak D_\infty=\mathcal A(\mathcal P)$ (see \cite{carbone}, corollary 2). 
\\

In the decoherence theory the set $\mathfrak{D}_\infty$ is called algebra of effective observables of our quantum dynamical system (see  e.g. \cite{blanchard03})  and we underline that the previous theorem is a particular case of a more general theorem that is found in \cite{lu-olk}. 
\\ 
We observe that for all natural number $n$ we obtain
$$\Phi^{\sharp n}(\Phi^n(d))=d \qquad  d\in\mathfrak{D}_\infty^+$$
and 
$$\Phi^n(\mathfrak{D}_\infty^+)\subset \mathfrak D_{\Phi^{\sharp n}}$$
We can say more:
\begin{remark}
The algebra of effective observables is independent by the stationary state $\varphi$, since
\begin{equation*}
\mathfrak{D}_\infty=\mathcal C_\Phi
\end{equation*}
In fact  we have that $\mathfrak{D}_\infty\subset\underset{n\in\mathbb{N}}{\bigcap}\Phi^n(\mathfrak{D}_\infty^+)$ since $\mathfrak{D}_\infty\subset\mathfrak{D}_\infty^+$ 
and $\mathcal C_\Phi \subset \mathfrak{D}_\infty$ for theorem \ref{automorphism}.
\end{remark}
The next subsections are of the simple consequences of the previous propositions.
\subsection{Ergodicity properties}
In this subsection we prove that the ergodic properties of a quantum dynamical system  depends  on its reversible part, determined from the algebra the effective observables $\mathfrak D_\infty$.
\\
We consider a quantum  dynamical system $( \mathfrak{M},\Phi,\varphi)$ with   $\varphi$-adjoint $\Phi^{\sharp}$.
\\
We recall that the quantum dynamical system is ergodic if per any $a,b\in\mathfrak{M}$ we have:
\[
\underset{N\rightarrow \infty }{\lim }\dfrac{1}{N+1}\sum\limits_{k=0}^{N} \left[ \varphi(a\Phi^k(b))-\varphi(a)\varphi(b) \right]=0
\]
while it is weakly mixing if
\[
\underset{N\rightarrow \infty }{\lim }\dfrac{1}{N+1}\sum\limits_{k=0}^{N} \left| \varphi(a\Phi^k(b))-\varphi(a)\varphi(b) \right|=0
\]
We will use again the following notations $a_\|=\mathcal{E}_{\infty}(a)$ while $a_\bot=a-a_\|$ for all $a\in\mathfrak{M}$, where $\mathcal{E}_{\infty}:\mathfrak{M}\rightarrow \mathfrak{D}_\infty $ is the conditional expectation of decomposition theorem \ref{teo-dec2}.
\\
We have the following proposition:
\begin{proposition}
The quantum dynamical system $( \mathfrak{M},\Phi,\varphi)$ is ergodic [\textit{weakly mixing}] if, and only if the reversible quantum dynamical system $( \mathfrak{D}_\infty,\Phi_\infty,\varphi_\infty)$ is ergodic  [\textit{weakly mixing}].
\end{proposition}
\begin{proof}
For any $a,b\in\mathfrak{M}$ we have 
\[
\varphi(a\Phi^k(b))-\varphi(a)\varphi(b)=\varphi(a\Phi^k(b_\|))+ \varphi(a\Phi^k(b_\bot))-\varphi(a_\|)\varphi(b_\|)
\]
Moreover $\underset{N\rightarrow \infty }{\lim }\dfrac{1}{N+1}\sum\limits_{k=0}^{N} \varphi(a\Phi^k(b_\bot))=0$, because by relation  (\ref{limitedebole}) for every $a\in\mathfrak{M}$, we have $\underset{k\rightarrow \infty }{\lim }\varphi(a\Phi^k(b_\bot))=0$, hence 
\begin{align*}
\underset{N\rightarrow \infty }{\lim }\dfrac{1}{N+1}\sum\limits_{k=0}^{N} \left[ \varphi(a\Phi^k(b)) - \varphi(a)\varphi(b) \right]=\underset{N\rightarrow \infty }{\lim }\dfrac{1}{N+1}\sum\limits_{k=0}^{N} \left[ \varphi(a\Phi^k(b_\|))-\varphi(a_\|)\varphi(b_\|) \right]
\end{align*}
with $\varphi(a\Phi^k(b_\|))=\varphi(a_\|\Phi^k(b_\|))+\varphi(a_\bot \Phi^k(b_\|))$ and $\varphi(a_\bot \Phi^k(b_\|))=0$ since the element $a_\bot \Phi^k(b_\|)\in\mathfrak{D}_\infty^{\perp_\varphi}$.\\
It follows that
\begin{align*}
\underset{N\rightarrow \infty }{\lim }\dfrac{1}{N+1}\sum\limits_{k=0}^{N} & \left[ \varphi(a\Phi^k(b)) - \varphi(a)\varphi(b) \right]
=\underset{N\rightarrow \infty }{\lim }\dfrac{1}{N+1}\sum\limits_{k=0}^{N} \left[ \varphi_\infty(a_\|\Phi_\infty ^k(b_\|))-\varphi_\infty(a_\|)\varphi_\infty(b_\|) \right]
\end{align*}
For the weakly mixing properties we have
\begin{align*}
&\underset{N\rightarrow \infty }{\lim }\dfrac{1}{N+1}\sum\limits_{k=0}^{N} \left| \varphi(a\Phi^k(b))-\varphi(a)\varphi(b) \right|= \\
&=\underset{N\rightarrow \infty }{\lim }\dfrac{1}{N+1}\sum\limits_{k=0}^{N} \left| \varphi(a_\|\Phi^k(b_\| )) + \varphi(a_\bot \Phi^k(b_\|))+ \varphi(a\Phi^k(b_\bot))-\varphi(a)\varphi(b) \right|=\\
&=\underset{N\rightarrow \infty }{\lim }\dfrac{1}{N+1}\sum\limits_{k=0}^{N} \left| \varphi_\infty(a_\|\Phi_\infty^k(b_\| )) -\varphi_\infty(a_\|)\varphi_\infty(b_\|) + \varphi(a\Phi^k(b_\bot)) \right|
\end{align*}
Moreover
\[
\underset{N\rightarrow \infty }{\lim }\dfrac{1}{N+1}\sum\limits_{k=0}^{N} | \varphi(a\Phi^k(b_\bot)) |=0, \qquad a,b\in\mathfrak{M}
\]
If our quantum dynamical system $( \mathfrak{M},\Phi,\varphi)$ is weakly ergodic then 
\begin{align*}
\underset{N\rightarrow \infty }{\lim }\dfrac{1}{N+1}\sum\limits_{k=0}^{N} \left| \varphi_\infty(a_\|\Phi_\infty^k(b_\| )) -\varphi_\infty(a_\|)\varphi_\infty(b_\|) \right|=0
\end{align*}
since
\begin{align*}
\left| \ | \varphi_\infty(a_\|\Phi_\infty^k(b_\| )) -\varphi_\infty(a_\|)\varphi_\infty(b_\|) |- | \varphi(a\Phi^k(b_\bot)) | \  \right|  \leq \left| \varphi(a\Phi^k(b))-\varphi(a)\varphi(b) \right|,
\end{align*}
while if the reversible dynamical system $( \mathfrak{D}_\infty,\Phi_\infty,\varphi_\infty)$ is weakly mixing, then our quantum dynamical system is weakly mixing since
\begin{align*}
\left| \varphi(a\Phi^k(b))-\varphi(a)\varphi(b) \right| \leq \left| \varphi_\infty(a_\|\Phi_\infty^k(b_\| )) -\varphi_\infty(a_\|)\varphi_\infty(b_\|) \right| + \left| \varphi(a\Phi^k(b_\bot)) \right|
\end{align*}
\end{proof}
\subsection{ Particular *-Banach algebra}
Let $(\mathfrak{M}, \Phi, \varphi )$ be a quantum dynamical system and $\mathcal E_\infty:\mathfrak M \rightarrow \mathfrak D_\infty$ the map  of proposition \ref{teo-dec2}.
\\
We can define in set $\mathfrak{M}$ another frame of *-Banach algebra changing the product between elements of $\mathfrak{M}$. It is defined by 
\begin{equation}
a\times b   =a_\| \ b_\| + a_\| \ b_\bot + a_\bot \  b_\|  \qquad \qquad a,b\in\mathfrak{M}
\end{equation}
where we have denoted with $a_\|=\mathcal{E}_{\infty}(a)$ and with $a_\bot=a-a_\|$ for all $a\in\mathfrak{M}$.
\\
We observe again that $a_\| \ b_\bot, a_\bot \  b_\| \in\mathfrak{D}_\infty^{\perp_\varphi}$ since $\mathcal{E}_{\infty}(a_\| \ b_\bot)=a_\| \mathcal{E}_{\infty}(\ b_\bot)=0$ and $ \mathcal{E}_{\infty}(a_\bot \  b_\|)=\mathcal{E}_{\infty}(a_\bot) \  b_\|=0$. Moreover we have 
$$ a_\bot \times b_\bot = 0 $$
The $(\mathfrak{M},+,\times )$ is a Banach *-algebra with unit, since for any $a,b\in\mathfrak{M}$ we have:
\[
||a \times b||\leq||a|| \ ||b||
\]
We set with $\mathfrak{M}^\flat$ this Banach *-algebra.
\\
We note that $\mathfrak{M}^\flat$ it is not a C*-algebra. In fact for any $d_\bot\in\mathfrak{D}_\infty^{\perp_\varphi}, \ \ d_\bot\neq0$ we have that its spectrum in $\mathfrak{M}^\flat$ is  $\sigma(d_\bot)\subset\{ 0\}$ while $ || d_\bot ||\neq0$.
\\
We observe that for any $a,b\in\mathfrak{M}$ we have:
\[
\Phi(a \times b)=\Phi(a) \times \Phi(b)
\]
It follows that $\Phi:\mathfrak{M}^\flat\rightarrow \mathfrak{M}^\flat$ is a *-homomorphism of Banach Algebra.
\\
For $\varphi$-adjoint $\Phi^\sharp$ we have:
\[
\varphi(a \times \Phi(b))=\varphi(a_\| \ \Phi(b_\|))=\varphi(\Phi^\sharp(a_\|) \ b_\|)= \varphi(\Phi^\sharp (a) \times b )
\]
with $\Phi^\sharp:\mathfrak{M}^\flat\rightarrow \mathfrak{M}^\flat$ *-homomorphism of Banach Algebra.
\\
Moreover $\varphi(a^*\times a)=\varphi(a_\|^* a_\|)$ hence if $\varphi(a^*\times a)=0$ then $ a_\|=0 $, so $\varphi$ it is not a faithful state on $\mathfrak{M}^\flat$.
\\
It is easily to prove that for any $a,b\in\mathfrak{M}^\flat$ we obtain $\varphi(a^* \times b^* \times b \times a) = \varphi(a^*_\| \ b^*_\| \ b_\| \ a_\|)$ it follows that
\[
\varphi(a^* \times b^* \times b \times a)\leq ||b|| \ \varphi(a^* \times a)
\]
and we can build the GNS representation $( \mathcal{H}_\varphi^\flat,\pi_\varphi^\flat, \Omega_\varphi^\flat)$ of the state $\varphi$ on Banach * algebra $\mathfrak{M}^\flat$ that has the following properties \cite{fell}:
\\
The representation $\pi^\flat_\varphi:\mathfrak{M}^\flat\rightarrow \mathfrak{B}(\mathcal{H}^\flat_\varphi)$ is a continuous map i.e.  $||\pi^\flat_\varphi(a)||\leq||a||$ for all $a\in\mathfrak{M}^\flat$ while  $\Omega^\flat_\varphi$ is a cyclic vector for *-algebra $\pi^\flat_\varphi(\mathfrak{M}^\flat)$ and
\[
\varphi(a)=\langle \Omega^\flat_\varphi, \pi^\flat_\varphi(a)\Omega^\flat_\varphi \rangle_\flat \qquad a\in\mathfrak{M}^\flat
\]
Furthermore we have a unitary operator $U^\flat_\varphi:\mathcal{H}_\varphi^\flat\rightarrow \mathcal{H}_\varphi^\flat$ such that 
\[
\pi_\varphi^\flat(\Phi(a)=U^\flat_\varphi \pi_\varphi^\flat(a) U^{\flat *}_\varphi \qquad a\in\mathfrak{M}^\flat
\]
since $\Phi$ and $\Phi^\sharp$ are *-homomorphism of Banach algebra and 
\begin{eqnarray*}
U^\flat_\varphi\pi_\varphi^\flat(a) \pi^\flat_\varphi(b)\Omega^\flat_\varphi=
\pi_\varphi^\flat(\Phi(a \times b)) \Omega^\flat_\varphi =
\pi_\varphi^\flat(\Phi(a )) \pi_\varphi^\flat(\Phi(b ))\Omega^\flat_\varphi= \pi_\varphi^\flat(\Phi(a)) U^\flat_\varphi\pi^\flat_\varphi(b)\Omega^\flat_\varphi  
\end{eqnarray*}
The linear map $Z :\mathcal{H}_\varphi^\flat \rightarrow \mathcal{H}_\varphi$ as defined 
$Z \pi^\flat_\varphi(a) \Omega_\varphi^\flat=\pi_\varphi ( \mathcal E_\infty (a) )\Omega_\varphi $ for all $a \in \mathfrak M $ it is an isometry with adjoint 
$Z^* \pi_\varphi(a)\Omega_\varphi=\pi_\varphi^\flat(\mathcal E_\infty(a)) \Omega_\varphi^\flat$ for all $a \in \mathfrak{M}^\flat$. 
\\
Furthermore  we have  $Z U_\varphi^{\flat n}   = Z  U_{\Phi, \varphi}^n $ for all natural number $n$.
\subsection{Abelian algebra of effective observables}
We will prove that for any quantum dynamical system $(\mathfrak{M}, \Phi, \varphi)$ there is an abelian algebra 
$\mathcal A \subset \mathfrak D_\infty$   that contains the center $Z(\mathfrak D_\infty)$ of $\mathfrak D_\infty$ and with $\Phi(\mathcal A)\subset \mathcal A$.
\\ 
The question of the existence of an Abelian subalgebra which remains invariant under the action of a given quantum Markov semigroup are widely debated in \cite{atreb} and \cite{rebolledo}.
\\
We consider a discrete quantum process $(\mathfrak{M}, \Phi)$ with $\Phi$ a *-automorphism. We set with $\mathfrak{P}(\mathfrak{M})$ the pure states of $\mathfrak{M}$.
\\
It is well know that if $\omega(a)=0$ for all $\omega\in\mathfrak{P}(\mathfrak{M})$ then $a=0$ (see e.g. \cite{black}).
\\
For any $\omega\in\mathfrak{P}(\mathfrak{M})$ with $\mathfrak{D}_\omega$ we set the multiplicative domain of the ucp-map $a\in\mathfrak{M}\rightarrow \omega(a)I\in\mathfrak{M}$, then
\[
\mathfrak{D}_\omega= \{ a\in\mathfrak{M}: \omega(a^*a)=\omega(a^*) \omega(a) \ \text{and}  \ \omega(aa^*)=\omega(a) \omega(a^*) \}
\]
it is a von Neumann subalgebra of $\mathfrak{M}$.
\begin{proposition} The von Neumann algebra 
\[
\mathcal{A}=\bigcap\{\mathfrak{D}_\omega : \omega\in\mathfrak{P}(\mathfrak{M}) \}
\]
is an abelian algebra with $\Phi(\mathcal{A})\subset\mathcal{A}$. Furthermore for any stationary state $\varphi$ of our quantum process $(\mathfrak M, \Phi)$, there is a $\varphi$-invariant conditional expectation $\mathcal{E}_\varphi:\mathfrak{M}\rightarrow\mathcal{A}$ such that 
\[
\mathcal{E}_\varphi \circ \Phi =\Phi
\]
\end{proposition}
\begin{proof}
If $a,b\in\mathcal{A}$, for any pure state $\omega$ of $\mathfrak{M}$ we have $\omega(ab)=\omega(a) \omega(b)=\omega(ba)$, then $\omega(ab-ba)=0$ and it follows that $ab-ba=0$.\\
The von Neumann algebra $\mathcal{A}$ is $\Phi$-invariant $\Phi(\mathcal{A})\subset\mathcal{A}$.\\
In fact $\omega \circ \Phi\in\mathfrak{P}(\mathfrak{M})$ for all $\omega\in\mathfrak{P}(\mathfrak{M})$ since $\Phi$ is a *-automorphism. Then for any $a\in\mathcal{A}$ we have
\[
\omega(\Phi(a^*) \Phi(a))=\omega(\Phi(a^*a))=\omega(\Phi(a^*))\omega(\Phi(a))
\]
it follows that $\Phi(a)\in\mathcal{A}$.\\
Let $\{ \sigma_\varphi^t \}_{t\in\mathbb{R}}$ be a modular group associate to GNS representation $( \mathcal{H}_\varphi,\pi_\varphi, \Omega_\varphi)$ of $\varphi$. Since the state $\varphi$ is normal and faithful we have $\pi_\varphi(\mathcal{A})''=\pi_\varphi(\mathcal{A})$ and $\sigma_\varphi^t(\pi_\varphi(\mathcal{A}))\subset\pi_\varphi(\mathcal{A})$ for all $t\in\mathbb{R}$.\\
In fact for any $a\in\mathcal{A}$ we have
\[
\omega(\sigma_\varphi^t(a^*)\sigma_\varphi^t(a))=\omega(\sigma_\varphi^t(a^*a))=\omega(\sigma_\varphi^t(a^*))\omega(\sigma_\varphi^t(a)) \ \qquad \omega\in\mathfrak{P}(\mathfrak{M})
\]
since $\sigma_\varphi^t$ is a *-automorphism so $\omega \circ \sigma_\varphi^t\in\mathfrak{P}(\mathfrak{M})$ for all real number $t$.\\
From Takesaki theorem \cite{takesaki72} we have that there is a conditional expectation $\mathcal{E}_\varphi:\mathfrak{M}\rightarrow\mathcal{A}$ such that
\[
\pi_\varphi(\mathcal{E}_\varphi(m))=\nabla^* \pi_\varphi(m) \nabla \ \qquad m\in\mathfrak{M}
\]
where $\nabla:\overline{\pi_\varphi(\mathcal{A})\Omega_\varphi}\longrightarrow \mathcal{H}_\varphi$ is the embedding map (see also \cite{accech}).
\end{proof}
We recall that any pure state is multiplicative on the center  $Z(\mathfrak M)=\mathfrak M \bigcap \mathfrak M'$ of $\mathfrak M$ (see \cite{moriyoshi}) so we have that $Z(\mathfrak M)\subset  \mathfrak{D}_\omega $ for all pure states $\omega$ and in abelian case  $\mathcal A=\mathfrak{M}$.
\\
 
Let $( \mathfrak{M},\Phi,\varphi)$ be a quantum dynamical system, with dynamics $\Phi$ that admits $\varphi$-adjoint $\Phi^{\sharp}$.\\
By the  decomposition theorem we have a *-automorphism $\Phi_\infty:\mathfrak{D}_\infty\rightarrow\mathfrak{D}_\infty$ with $\mathfrak{D}_\infty$ von Neumann algebra, then by the previous proposition, we can say that there exist an abelian algebra $\mathcal{A}\subset\mathfrak{D}_\infty$ with $\Phi(\mathcal{A})\subset\mathcal{A}$ getting the following commutative diagram  
\[
\begin{array}
[c]{ccccc}
& \mathfrak{M} & \overset{\Phi}{\longrightarrow} & \mathfrak{M}  & \\
i_\infty & \uparrow &  & \downarrow &\mathcal{E}_\infty \\
& \mathfrak{D}_\infty & \overset{\Phi_\infty}{\longrightarrow} & \mathfrak{D}_\infty &  \\
i_o& \uparrow &  & \downarrow & \mathcal{E}_\varphi\\
&  \mathcal{A}  & \overset{\Phi_o}{\longrightarrow} & \mathcal{A} &
\end{array}
\]
where $i_\infty$ and $i_o$ are the embeddig of $\mathfrak{D}_\infty$ and $\mathcal{A}$ respectively, while $\Phi_\infty$ and $\Phi_o$ are the restriction of $\Phi$ to $\mathfrak{D}_\infty$ and $\mathcal{A}$ respectively.\\
We observe that if the von Neumann algebra $\mathfrak{M}$ is abelian then $\mathcal{A}=\mathfrak{D}_\infty$.
\\ 
\subsection{Dilation properties}
We recall that a reversible quantum dynamical system $(\widehat{\mathfrak{M}},\widehat{\Phi},\widehat{\varphi})$, is said to be a dilation of the quantum dynamical system $(\mathfrak{M},\Phi,\varphi)$, if it satisfies the following conditions:
\\
There is *-monomorphism  $i:( \mathfrak{M}, \varphi)\rightarrow (\widehat{\mathfrak{M}}, \widehat{\varphi})$ and a completely positive map 
$\mathcal E:\widehat{\mathfrak{M}} \rightarrow \mathfrak{M}$ such that for each $a$ belong to $\mathfrak{M}$ and natural number $n$
\[
\mathcal E(\widehat{\Phi }^{n}(i(a)))=\Phi ^{n}((a))
\]
We observe that for each $a$ belong to $\mathfrak{M}$ and $X$ in $\widehat{\mathfrak{M}}$ we have:
\begin{equation*}
\mathcal E(i(a)X)=a \mathcal E (X).
\end{equation*}
Indeed for each $b\in\mathfrak{M}$ we obtain:
\[
\varphi(b \mathcal E(i(a) X)=\varphi(i(b) i(a) X)=\varphi(i(b a)X)=\varphi(b a \mathcal E(X))
\]
So, the ucp-map $\widehat{\mathcal{E}} = i\circ \mathcal E$ is a conditional expectation from  $\widehat{\mathfrak{M}}$ onto $i(\mathfrak{M})$ which leave invariant a faithful normal state.
The existence of such map which characterize the range of existence of a reversible dilation of a dynamical system, be derived from a theorem of Takesaki of \cite{takesaki72}.
\\ 
\indent
We have a proposition that establish a link between the algebra of effective observable and reversible dilation.
\begin{proposition}
If $(\widehat{\mathfrak{M}},\widehat{\Phi},\widehat{\varphi})$ is a dilation of quantum dynamical system $(\mathfrak{M},\Phi,\varphi)$ then 
\[
\widehat{\Phi}(i(a)=i(\Phi(a)) \quad \textsl{if, and only if } \quad a\in\mathfrak{D}_\Phi
\]
\end{proposition}
\begin{proof}
We have $i(\Phi(a)^*) \ i(\Phi(a))=\widehat{\Phi}(i(a)^*) \widehat{\Phi}(i(a))$ it follows that 
\[
\Phi(a^*) \ \Phi(a)= \mathcal E (i(\Phi(a)^*\Phi(a)))= \mathcal E(\widehat{\Phi}(i(a^*a))=\Phi(a^* a).
\]
For vice-versa, if $y=i(\Phi(a))- \widehat{\Phi}(i(a))$ then we have
\[
y^*y= i(\Phi(a^*a))-\widehat\Phi(i(a^*) i(\Phi(a))- i(\Phi(a^*))\widehat\Phi(i(a))+ \widehat\Phi(i(a^*a))
\]
since $a\in\mathfrak{D}_\Phi$. It follows that $\mathcal E(y^*y)=0$ with $\mathcal E$ faithful map, then $y=0$.
\end{proof}
Let $\mathfrak{M}=\mathfrak{D}_\infty \oplus \mathfrak{D}_\infty^{\perp_\varphi}$ be decomposition of theorem \ref{teo-dec} of our quantum dynamical system $(\mathfrak{M},\Phi,\varphi)$ and $\mathcal{E}_\infty : \mathfrak{M} \rightarrow \mathfrak{D}_\infty$ the conditional expectation defined in proposition \ref{teo-dec2}, we can say:
\begin{remark}
For each $a\in\mathfrak{M}$ and integer $k$ we have:
\[
\widehat{\Phi}^k(i(\mathcal{E}_\infty(a)))=i(\Phi_k (\mathcal{E}_\infty(a))
\]
\end{remark}
We observe that 
\[
X\in\ \textit{i}(\mathfrak{D}_\infty)^{\perp_{\widehat{\varphi}}} \quad \textsl{if , and only if }  \quad   \mathcal E(X)\in\mathfrak{D}_\infty^{\perp_\varphi}
\]
since $i(\mathfrak{D}_\infty)^{\perp_{\widehat{\varphi}}} = \{X\in\widehat{\mathfrak{M}}: \widehat{\varphi}(X^* i(d))=0 \quad \forall d\in\mathfrak{D}_\infty \}$ and 
$ \widehat{\varphi}(X^* i(d))=\varphi(\mathcal E(X^*)d)$ for all $d\in\mathfrak{D}_\infty$.
\\
We can write the following algebraic decomposition of linear spaces:
\[
\widehat{\mathfrak{M}}=\textit{i}(\mathfrak{D}_\infty) \oplus \textit{i}(\mathfrak{D}_\infty)^{\perp_{\widehat{\varphi}}}
\]
and the ucp-map $ \widehat{\mathcal{E}}_\infty =i \circ \mathcal{E}_\infty \circ \mathcal E$ is a conditional expectation from  $\widehat{\mathfrak{M}}$ onto $i(\mathfrak{D}_\infty)$.
\section{ Decomposition theorem and linear contractions}
We would study the relations between canonical decomposition of Nagy-Fojas of linear contraction $U_{\Phi,\varphi}$ \cite{nagy-foias} and decomposition (c) of   proposition \ref{teo-dec2}  of dynamical system $(\mathfrak{M},\Phi, \varphi)$.
\\
We going to recall the main statements  of these topics. 
\\

A contraction $T$ on the Hilbert space $\mathcal{H}$ is called \textit{completely non-unitary} if for no non zero reducing subspace $\mathcal{K}$ for $T$ is $T_{\mid\mathcal{K}}$ a unitary operator, where $T_{\mid\mathcal{K}}$ is the restriction of contraction $T$ on the Hilbert space $\mathcal{K}$.\\
We set with $D_T=\sqrt{I-T^*T}$ the defect operator of the contraction $T$ and it is well know that
\[
TD_T=D_{T^*} T 
\]
Moreover $ ||T\psi|||=||\psi|| $ if, and only if $ D_T\psi=0$.
\\
We consider the following Hilbert subspace of $\mathcal{H}$:
\begin{equation}
\mathcal{H}_0= \{ \psi\in\mathcal{H}: ||T^n\psi||=||\psi||=||T^{* n}\psi|| \ for \ all \ \ n\in\mathbb{N} \} 
\end{equation}
It is trivial show that $T^n\mathcal{H}_0=\mathcal{H}_0$ and $T^{* n}\mathcal{H}_0=\mathcal{H}_0$ for all natural number $n$.
\\

We have the following canonical decomposition (see \cite{nagy-foias}):
\begin{theorem}[\textit{Sz-Nagy and Fojas}]
\label{Sz-Nagy and Fojas} 
To every contraction $T$ on $\mathcal{H}$ there corresponds a uniquely determined decomposition of $\mathcal{H}$ into a orthogonal sum of two
subspace reducing $T$ we say $\mathcal{H=H}_{0}\mathcal{\oplus H}_{1}$, such that $T_0=T_{\mid\mathcal{H}_{0}}$ is unitary and $T_1=T_{\mid\mathcal{H}_{1}}$ is c.n.u., where
\[
\mathcal{H}_{0}=\bigcap\limits_{k\in\mathbb{Z}}\ker\left(  D_{T_{k}}\right)  \qquad and \qquad \mathcal{H}_{1}=\mathcal{H}_{0}^{\perp}
\]
with
\[
T_{k}=
\left\{
\begin{array}
[c]{cc}
T^k \ & k\geq0 \\
 T^{* -k  } & k<0
\end{array}
\right .
\]
\end{theorem}
It is well know  \cite{nagy-foias} that  the linear operator $T_{-}= so-\underset{n\rightarrow+\infty}{\lim}T^{* n}T^n $ and $T_{+}=so-\underset{n\rightarrow+\infty}{\lim}T^n T^{* n} $, there are in sense of strong operator ($so$) convergence.
\\
After this brief detour on linear contractions we return to quantum dynamical systems $( \mathfrak{M},\Phi,\varphi)$.
\\
We set  $V_{-}=so-\underset{n\rightarrow+\infty}{\lim}U_{\Phi,\varphi}^{* n}U_{\Phi,\varphi}^n $ and $V_{+}=so-\underset{n\rightarrow+\infty}{\lim}U_{\Phi,\varphi}^n U_{\Phi,\varphi}^{* n} $,  where $U_{\Phi,\varphi}$ is the contraction defined in (\ref{contrazione 1}).
\\
It follows that for each $a,b\in\mathfrak{M}$ we obtain:
\begin{equation}
\label{limitecontra}
\underset{n\rightarrow \pm\infty}{\lim}\varphi(S_{n}(a,b))=\langle \pi_{\varphi}(a)\Omega_\varphi ,(I-V_\pm )\pi_{\varphi}(b) \Omega_\varphi\rangle
\end{equation}
where $S_{n}(a,b)$ is given by (\ref{sesqui}).
\\
We recall that by proposition \ref{teo-dec2} that for every integers $k$ we obtain $\mathcal H_\varphi=\mathcal H_\infty \oplus \mathcal K_\infty$ with $U_k\mathcal H_\infty = \mathcal H_\infty$ and $U_k \mathcal K_\infty \subset \mathcal K_\infty$, where
\[
U_{k}=
\left\{
\begin{array}
[c]{cc}
U_{\Phi,\varphi}^k & k\geq0 \\

 U_{\Phi,\varphi}^{* -k } & k<0
\end{array}
\right .
\]
A simple consequences of proposition \ref{prop0} is the following remark:
\\
For any integers $k$ we obtain
$$ a\in \mathfrak D_{\Phi_k} \quad  \textsl{if, and only if} \quad \pi_\varphi(a)\Omega_\varphi \in \ker ( D_{U_k}) \ \textsl{and} \ \pi_\varphi(a^*)\Omega_\varphi \in \ker (D_{U_k})$$
\\
Therefore $\mathcal H_\infty \subset  \mathcal H_0 $ because
$$ \pi_\varphi(\mathcal D_\infty)\Omega_\varphi \subset \bigcap_{k\in\mathbb Z} \pi_\varphi(\mathfrak D_{\Phi_k}) \Omega_\varphi \subset \bigcap\limits_{k\in\mathbb{Z}}\ker (  D_{U_k})$$
We observe that for each $a,b\in\mathfrak{M}$ and natural number $k$ we have (see \cite{frigerio} theorem 3.1):
\begin{equation}
\label{limite1}
\underset{n\rightarrow+\infty}{\lim}\varphi(S_{k}(\Phi^n(a),b))=0
\end{equation}
Indeed for each natural numbers $k$ and $n$, we obtain
\[
\varphi(S_{k}(\Phi^n(a),\Phi^n(b))=\varphi(S_{k+n}(a,b))-\varphi(S_{n}(a,b))
\]
and by the relation  (\ref{limitecontra}) result
$ \underset{n\rightarrow+\infty}{\lim}(\varphi(S_{k+n}(a,b))-\varphi(S_{n}(a,b)))=0 $.
\\
Furthermore, for each natural number $k$ and $a,b\in\mathfrak{M}$ we have 
\[
| \varphi(S_{k}(\Phi^n(a),b))|^2\leq\varphi(S_{k}(\Phi^n(a),\Phi^n(a)) \varphi(S_{k}(b,b)) 
\]
it follows that $\underset{n\rightarrow+\infty}{\lim}\varphi(S_{k}(\Phi^n(a),b)=0$.
\\
We have a well-known statement (see \cite{Hellmich}, \cite{lu-olk} and \cite{stormer07} ):
\begin{proposition}
For all $a\in\mathfrak{M}$ any $\sigma$-limit point of the set $\{ \Phi^k(a) \}_{k\in\mathbb{N}}$ belongs to the von Neumann algebra $\mathfrak{D}_{\infty}$. Moreover, for each $d_\bot$ in $\mathfrak{D}_{\infty}^{\perp_{\varphi}}$ we have:
\[
\underset{k\rightarrow+\infty}{\lim}\Phi^k(d_\bot)=0 \qquad and  \qquad \underset{k\rightarrow+\infty}{\lim}\Phi^{\sharp k}(d_\bot)=0
\]
where the limits are in $\sigma$-topology.
\end{proposition}
\begin{proof}
If $y$ is a $\sigma$-limit point of $\{ \Phi^n(a) \}_{n\in\mathbb{N}}$ then there exists a net $\{ \Phi^{n_{j}}(a) \}_{j\in\mathbb{N}}$ such that
 $y=\underset{j\rightarrow+\infty}{\lim}\Phi^{n_{j}}(a)$ in $\sigma$-topology. Furthermore for each $b\in\mathfrak{M}$ we obtain 
\[
S_{k}(y,b)=\sigma-\underset{j\rightarrow+\infty}{\lim} [ \ \Phi^k(\Phi^{n_j}(a)b)-\Phi^k(\Phi^{n_j}(a))\Phi^k(b) \ ] =\sigma-\underset{j\rightarrow+\infty}{\lim}S_{k}(\Phi^{n_{j}}(a),b)
\]
from (\ref{limite1}) we obtain $\underset{j\rightarrow+\infty}{\lim}\varphi(S_{k}(\Phi^{n_{j}}(a),b))=0$ hence $\varphi(S_{k}(y,b))=0$.
\\
It follows that $\varphi(S_{k}(y,y))=0$ and $S_{k}(y,y)=0$.
\\
We observe that the adjoint is $\sigma$-continuous, then we obtain $y^*=\underset{j\rightarrow+\infty}{\lim}\Phi^{n_{j}}(a^*)$,  and repeating the previous steps we obtain $S_{k}(y^*,y^*)=0$, hence $y\in\mathfrak{D}_{\infty}$.
\\
For last statement we observe that for each natural number $k$ result $||\Phi^k(d_\bot)||\leq ||d_\bot||$ and since the unit ball of the von Neumann algebra $\mathfrak{M}$ is $\sigma$- compact we have that there is a subnet such that $\Phi^{k_\alpha}(d_\bot)\rightarrow y\in\mathfrak{D}_{\infty}^{\perp_{\varphi}}$ in $\sigma$-topology. From previous lemma we have that $y\in\mathfrak{D}_{\infty}\cap\mathfrak{D}_{\infty}^{\perp_{\varphi}}$ it follows that $y=0$. 
then it can only be $\underset{k\rightarrow+\infty}{\lim}\Phi^k(d_\bot)=0$ in $\sigma$-topology.
\end{proof}
We observed that the Hilbert space $\mathcal H_\infty$, the linear closure of  $\pi_\varphi (\mathfrak D_\infty)\Omega_\varphi$ is contained in $\mathcal H_0$. The next step is to understand when we have the equality of these two Hilbert spaces.
\\

Let $( \mathfrak{M},\Phi,\varphi)$ be the previous quantum dynamical system, we define, for each integer $k$, the unital Schwartz   map $\tau_k:\mathfrak{M} \rightarrow \mathfrak{M}$ as 
\begin{equation}
\label{taumap}
\tau_k=\Phi_{-k}\circ \Phi_k \qquad \qquad k\in \mathbb Z
\end{equation}
We have for any integer $k$ that 
\begin{itemize}
\item [1 - ]  $\varphi \circ \tau_k = \varphi$
\item [2 - ]  $\tau_k=\tau_k^\sharp$, where $\tau_k^\sharp$ is the $\varphi$-adjoint of $\tau_k$.
\end{itemize}
We obtain, for any integer $k$ the  dynamical system $\left\{\mathfrak M, \tau_k, \varphi \right\}$ with
\[
\mathfrak D_\infty (\tau_k)= \bigcap_{j\geq 0} \mathfrak D(\tau_k^j)
\]
where with $ \mathfrak D(\tau_k^j) $ we have denote the multiplicative domains of map $\tau_k ^j$.
\\
From decomposition theorem \ref{teo-dec}, for any integer $k$ we have: 
\[
\mathfrak M= \mathfrak D_\infty (\tau_k) \oplus \mathfrak D_\infty (\tau_k)^ {\perp_\varphi}
\]
and by the  proposition \ref{hilbertdeco2} 
\[
\mathcal H_\varphi = \mathcal H_{(k)} \oplus \mathcal K_{(k)}
\]
where $\mathcal H_{(k)} $ and $\mathcal K_{(k)}$ are the closure of the linear space $\pi_\varphi (\mathfrak D_\infty (\tau_k) )\Omega_\varphi$ and of 
$\pi_{\varphi}(\mathfrak D_\infty (\tau_k) )^{\perp_{\varphi}})\Omega_{\varphi}$ respectively.
\\
We have the following proposition:
\begin{proposition}
If $ \overline{\pi_\varphi (\mathfrak D_{\Phi_k})\Omega_\varphi}$ denotes the closure of  linear space $\pi_\varphi (\mathfrak D_{\Phi_k})\Omega_\varphi$ then we have 
\[
\mathcal H_0 = \bigcap_{k\in\mathbb Z} \overline{\pi_\varphi (\mathfrak D_{\Phi_k})\Omega_\varphi}
\]
where the $\mathcal H_0$  is the Hilbert space of Nagy decomposition of theorem \ref{Sz-Nagy and Fojas}.
\\
Furthermore for any $a\in\mathfrak M$ and $\xi_0 \in \mathcal H_0$ and integr $k$ we have
\begin{equation}
\label{equationdelta2}
U_ {\Phi,\varphi}^k \pi_\varphi(a)  \xi_0 =\pi_\varphi(\Phi^k (a)) U_{\Phi, \varphi}^k  \xi_0
\end{equation}
\end{proposition}
\begin{proof}
We have that $\mathfrak D(\tau_k)\subset \mathfrak D_{\Phi_k}$ for all integers $k$. 
\\
In fact if $a\in\mathfrak D(\tau_k)$ then 
\begin{eqnarray*}
\varphi(\Phi_k(a^* a) ) & = & \varphi (a^* a ) = \varphi(\tau_k(a^* a))= \varphi (\tau_k(a^*) \tau_k (a)) =\varphi(\Phi_{-k}(\Phi_k(a)^*) \Phi_{-k}(\Phi_k(a))) \leq
\\
&\leq & \varphi(\Phi_{-k}(\Phi_k(a)^* \Phi_k(a)))=\varphi(\Phi_k(a^*)\Phi_k(a))\leq \varphi(\Phi_k(a^*a)) 
\end{eqnarray*}
It follows that $\varphi(S_k(a,a))=0$ for all integers $k$ and in the same way proves that $\varphi(S_k(a^*,a^*))=0$ for all integers $k$. 
\\
We have proved that 
\[
\mathfrak D_\infty(\tau_k) =\bigcap_{j\in\ \mathbb N} \mathfrak D_{\tau_k^j} \subset \mathfrak D_{\tau_k} \subset \mathfrak D_{\Phi_k}
\]
If $\xi_0\in\mathcal H_0$  then for any $k$ integer and natural number $n$ we have $ (U_{\Phi, \varphi,}^{* k} U_{\Phi,\varphi}^{k})^n \xi_0=\xi_0$ and for any $r_\bot \in \mathfrak{D}_\infty(\tau_k)^{\perp_\varphi}$ we can write that 
\begin{equation*}
\left\langle\pi_\varphi(r_\bot)\Omega_\varphi, \xi_0 \right\rangle = \left\langle (U_{\Phi, \varphi,}^{* k} U_{\Phi,\varphi}^{k} )^n \ \pi_\varphi(r_\bot)\Omega_\varphi, \xi_0 \right\rangle = 
 \left\langle \pi_\varphi(\tau_k^n(r_\bot)\Omega_\varphi, \xi_0 \right\rangle   
\end{equation*}
and
\[
\lim_{n \rightarrow + \infty}  \left\langle \pi_\varphi(\tau_k^n(r_\bot)\Omega_\varphi, \xi_0 \right\rangle  =0 \qquad k\in \mathbb Z
\]
since $\tau_k^n(r_\bot) \longrightarrow  0 $ as $ n \rightarrow \infty $ in $\sigma$-topology. 
\\
It follows that $\mathcal H_0 \subset [ \pi_\varphi (\mathfrak{D}_\infty(\tau_k)^{\perp_\varphi})\Omega_\varphi ] ^ \perp = [ \mathcal K_k  ] ^\perp$. 
\\
Therefore for any integers $k$ we obtain:
\[
\mathcal H_0 \subset \mathcal H_{(k)}  \subset \overline{\pi_\varphi (\mathfrak D_{\Phi_k})\Omega_\varphi} \qquad  \Longrightarrow  \qquad \mathcal H_0 \subset\bigcap_{k\in\mathbb Z} \overline{\pi_\varphi (\mathfrak D_{\Phi_k})\Omega_\varphi}
\]
Let $\xi_0\in \bigcap_{k\in\mathbb Z} \overline{\pi_\varphi (\mathfrak D_{\Phi_k})\Omega_\varphi} $,  for any integers $k$ we have a net 
$d_{\alpha,k}\in \mathfrak D_{\Phi_k}$ such that $\pi_\varphi (d_{\alpha,k})\Omega_\varphi \rightarrow \xi_0$ as $\alpha\rightarrow \infty$ and for $k\geq 0$ we obtain
$$U^{* k}_{\Phi,\varphi}U^{ k}_{\Phi,\varphi} \xi_0 =
U^{* k}_{\Phi,\varphi}U^{ k}_{\Phi,\varphi} \lim_\alpha \pi_\varphi (d_{\alpha,k})\Omega_\varphi= 
 \lim_\alpha U^{* k}_{\Phi,\varphi}U^{ k}_{\Phi,\varphi}  \pi_\varphi (d_{\alpha,k})\Omega_\varphi=
 \lim_\alpha \pi_\varphi (d_{\alpha,k})\Omega_\varphi=\xi_0 $$
in the same way for $k\geq 0$ we have $U^{ k}_{\Phi,\varphi} U^{* k}_{\Phi,\varphi} \xi_0=\xi_0$. 
\\
It follows that 
$$ \bigcap_{k\in\mathbb Z} \overline{\pi_\varphi (\mathfrak D_{\Phi_k})\Omega_\varphi}\subset \mathcal H_0 $$
The relation (\ref{equationdelta2}) is a straightforward.
\end{proof}
We observe that  for any $a\in\mathfrak{M}$ and $d_\bot\in\mathfrak{D}_\infty^{\perp_\varphi}$ we have
\begin{equation}
\label{limitedebole}
\underset{n\rightarrow\infty}{\lim}\varphi(a^*\Phi_n(d_\bot)a)=0
\end{equation}
since for any $d_\bot\in\mathfrak{D}_{\infty}^{\perp_{\varphi}}$ we obtain $\Phi_n(d_\bot)\rightarrow 0$ as $n\rightarrow \infty$ in $\sigma$-topology.
\\
From polarization identity we can say that
\[
\underset{n\rightarrow\infty}{\lim}\varphi(a\Phi_n(d_\bot)b)=0, \qquad \ a,b\in\mathfrak{M}, \ d_\bot\in\mathfrak{D}_\infty^{\perp_\varphi}
\]
and since $U_\varphi$ is a contraction it follows that for any $\xi\in\mathcal{H}_\varphi$ and $\psi\in\mathcal{K}_\infty$ we have
\begin{equation}
\label{limitedeboleU}
\underset{n\rightarrow\infty}{\lim}\langle \xi, U_{\Phi,\varphi}^n \psi \rangle=0 \qquad \text{and} \qquad \underset{n\rightarrow\infty}{\lim}\langle \xi, U_{\Phi,\varphi}^{*n} \psi \rangle=0
\end{equation}
We give a simple statement on the Hilbert spaces $\mathcal H_\infty$ and $\mathcal H_0$: 
\begin{proposition}
\label{s-conv}
If  $\Phi^n(d_\bot)\rightarrow 0$  $ [ \Phi^{\sharp n}(d_\bot)\rightarrow 0 ] $ \ as \ $n\rightarrow \infty$ in $s$-topology for all $d_\bot\in\mathfrak{D}_\infty^{\perp_\varphi}$; then  $\mathcal{H}_\infty=\mathcal{H}_0$ and $V_+ = P_\infty$ $ [ V_- = P_\infty ]$
\end{proposition}
\begin{proof}
We observe that for any $\psi\in\mathcal{K}_\infty$ result  $||U_{\Phi,\varphi}^n \psi||\rightarrow 0$ \ as $n\rightarrow \infty$, because for any $k\in\mathbb N$ there is $d_k^\bot\in\mathfrak{D}_\infty^{\perp_\varphi}$ such that $|| \psi - \pi_\varphi(d_k^\bot)\Omega_\varphi|| < 1/k$ and $U_{\Phi,\varphi}^n$ is a linear contraction so for all natural number $n$ we obtain:
$$||U_{\Phi,\varphi}^n \psi|| < \frac{1}{k}+\varphi(\Phi^n(d_k^\bot)^* \Phi^n(d_k^\bot))$$
If $\xi_0\in\mathcal{H}_0$, we can write $\xi_0=\xi_\| + \xi_\bot$ with $\xi_\| \in\mathcal{H}_\infty$ and $\xi_\bot\in\mathcal{K}_\infty$. 
\\
Then $\xi_\bot=\xi_0 - \xi_\| \in \mathcal H_0$ therefore 
$$|| \xi_\| || + || \xi_\bot|| = || U_{\Phi, \varphi}^n \xi_0 ||=  ||U_{\Phi, \varphi}^n \xi_\| +  U_{\Phi, \varphi}^n \xi_\bot ||= ||\xi_\| || +  ||U_{\Phi, \varphi}^n \xi_\bot || $$
for all natural numbers $n$ it follows that $\xi_\bot =0$. 
\\
Moreover for any $\xi\in\mathcal{H}_\varphi$ we have  $U_{\Phi,\varphi}^{n *} U_{\Phi,\varphi}^n \xi= \xi_0 + U_{\Phi,\varphi}^{n *} U_{\Phi,\varphi}^n\xi_1$ with $\xi_i\in\mathcal H_i$ for $i=1,2$ and 
$V_+ \xi=\xi_0$ since $||U_{\Phi,\varphi}^n \xi_1||\rightarrow 0$ \ as $n\rightarrow \infty$. 
\end{proof}
We conclude this section with a simple observation: 
\\
We recall that a dynamical system $\{ \mathfrak{M}, \Phi, \omega \}$ is mixing if
\begin{equation}
\lim_{n\rightarrow \infty} \varphi(a\Phi^n(b))=\varphi(a) \varphi(b) \ , \qquad a,b\in\mathfrak M
\end{equation}
 by the relation (\ref{limitedebole}) we obtain that $\{ \mathfrak{M}, \Phi, \omega \}$ is mixing if, and only if its reversible part $(\mathfrak{D}_{\infty},\Phi_{\infty},\varphi_{\infty})$ is mixing.
\\
Furthermore, let $\{ \mathfrak{M}, \Phi, \omega \}$ be a mixing Abelian dynamical system, then there is a measurable dynamics space 
$(X, \mathcal A, \mu,  T)$ such that $\mathfrak{D}_{\infty}$  is isomorphic to the von Neumann algebra $L^\infty(X, \mathcal A, \mu)$ of the measurable bounded function on $X$. If the set $X$ is a metric space and $ \varphi_{\infty}$ is the unique staionary state of $\mathfrak{D}_{\infty}$ for the dynamics $\Phi_{\infty}$, then by the corollary 4.3 of \cite{fidaleo} we have $\mathfrak{D}_{\infty}=\mathbb C 1$. 
\section{ Decomposition theorem and Cesaro mean }
In this section we will study the link between the decomposition theorem \ref{teo-dec} and some ergodic results which we recall briefly. 
\\
It is well known the following proposition (see e.g. \cite{krengel} par. 9.1 and \cite{mohari} proposition 2.3). 
\begin{proposition}
\label{ec1}
Let $\{ \mathfrak{M}, \tau, \omega \}$ be a quantum dynamical system. We consider the Cesaro mean 
\begin{equation*}
s_n=\frac{1}{n+1}\sum_{k=0}^n{\tau^k},
\end{equation*}
Then, there is an $\omega$-conditional expectation $\mathcal E$ of  $\mathfrak{M}$ onto fixed point $\mathcal{F}(\tau)=\left\{a\in\mathfrak M : \tau(a)=a \right\}$ such that 
\[
\lim_{n\rightarrow 0} ||\phi \circ s_n -\phi \circ \mathcal E ||=0 \qquad \phi\in\mathfrak{M}_*
\]
\end{proposition}
A simple consequence of the previous proposition is the following remark:
\begin{remark}
$ \left\{ \mathfrak{M}, \tau, \omega \right\}$ is ergodic if, and only if $ \mathcal{F}(\tau)=\mathbb{C}1$
\end{remark}
Let $( \mathfrak{M},\Phi,\varphi)$ be the previous quantum dynamical system, and $\tau_k:\mathfrak{M} \rightarrow \mathfrak{M}$ the Schwartz map defined in (\ref{taumap}), we have a simple statement:
\begin{proposition}
For each integer $k$ we obtain:
\[
\mathcal{F}(\tau_k)=\mathfrak{D}_{\Phi_k}
\]
\end{proposition}
\begin{proof}
Without loss of generality we assume $k=1$ then $\tau_1=\Phi^\sharp \circ \Phi$.\\
If $x\in\mathcal{F}(\tau_1)$ we can write $\varphi(\Phi(x^*) \Phi(x))=\varphi(x^*\tau_1(x))=\varphi(x^*x)=\varphi(\Phi(x^*x))$ then $x\in\mathfrak{D}_\Phi$. The converse is proved similarly.
\end{proof}
Now let us ask when the algebra of effectives observables $\mathfrak{D}_\infty$ is trivial (see also \cite{carbone2} proposition 15) . 
\begin{proposition}
If $\mathcal{D}_\infty=\mathbb{C}1$ then the normal state $\varphi$ is of asymptotic equilibrium and the quantum dynamical system $( \mathfrak{M},\Phi,\varphi)$ is ergodic.
\end{proposition}
\begin{proof}
By decomposition theorem $\mathfrak{M}=\mathbb{C}1 \oplus\mathfrak{D}_\infty^{\perp_\varphi}$ and for each $a\in\mathfrak{M}$  we have $a=\varphi(a)1+a_\perp$ with $a_\perp\in\mathfrak{D}_\infty^{\perp_\varphi}$. 
It follows that 
$$\Phi^n(a)= \varphi(a) 1 + \Phi^n(a_\perp)$$
and  $\Phi^n(a_\perp)\rightarrow 0$ in $\sigma$-top.
\end{proof}
We have a simple consequence of the previous propositions:
\begin{corollary}
If the quantum dynamical system $\{ \mathfrak{M}, \tau_k, \varphi \}$ is ergodic for some integer $k$, then $\mathfrak{D}_{\infty}=\mathbb{C}1$.
\end{corollary}
\begin{proof}
If we have ergodicity then $\mathcal{F}(\tau_k)=\mathfrak{D}_{\Phi_k}=\mathbb{C}1$.
\end{proof}
Summarizing 
$$\tau_1 \ ergodic \quad  \Longrightarrow  \quad \Phi \  completely \ irreversible \quad \Longrightarrow  \quad \Phi \  ergodic$$
We observe that if  $( \mathfrak{M},\Phi,\varphi)$ is a quantum dynamical system with $\Phi$ homomorphism, we have that $\tau_1=\Phi^\sharp \circ \Phi=id$. Hence the dynamical system $\{ \mathfrak{M}, \tau_1, \varphi \}$ is not ergodic (if $\varphi$ is not multiplicative functional), while $( \mathfrak{M},\Phi,\varphi)$ can be.
\\

For each integer $k$ we consider $S_{n,k}=\frac{1}{n+1}\sum \limits_{j=0}^{n} \tau^j_k$.
\\
By previous proposition \ref{ec1} there is a positive map $\mathcal E_k:\mathfrak{M} \rightarrow \mathfrak{M}$ such that
\[
   ||\phi \circ S_{n,k} -\phi \circ \mathcal E_k|| \rightarrow 0 \qquad \phi\in\mathfrak{M}_* 
\]
and $\mathcal E_k$ is the conditional expectation related of von Neumann algebra $\mathcal{D}_{\Phi_k}$ of theorem \ref{teo-dec}.
\\
Therefore $\mathcal E_k :\mathfrak M \rightarrow \mathcal{D}_{\Phi^k}$  and $\varphi \circ \mathcal E_k= \varphi$ for all integers number $k$.
\\
Furthermore we have: 
$$ \mathcal E_h \circ \mathcal E_k= \mathcal E_k \qquad k\geq h \geq 0 $$  
because by the relation \ref{inclusion} we have $\mathcal{D}_{\Phi^k}\subset \mathcal{D}_{\Phi^h}$  for all $k\geq h$.
\\
For each $a\in\mathfrak M$ we have $||\mathcal E_k (a)||\leq ||a||$ for all integers $k$ and apply $\sigma$-compactness property for the bounded net  $ \left\{ \mathcal E_{k} (a) \right\}_{k\in\mathbb N}$ of von Neumann algebra $\mathfrak M$,  we obtain that there is at lest one $\sigma$-limit point $\mathcal E_+ (a)$, therefore there exist a net $ \left\{ \mathcal E_{n_\alpha} (a) \right\}_\alpha$ such that $\mathcal E_+(a)=\sigma-\lim_\alpha \mathcal E_{n_\alpha}(a)$.
\\
We obtain that $\mathcal E_+(a)\in\mathcal{D}_{\Phi^k}$ for all natural number $k$ because for any $a\in\mathfrak M$ we have 
$\mathcal E_h (\mathcal E_{n_\alpha}(a))= \mathcal E_{n_\alpha}(a)$
 when  $n_\alpha\geq h$ and since $\mathcal E_h$ are normal maps follows that
$$\mathcal E_h (  \mathcal E_+(a) ) = \mathcal E_+(a)$$
for all  natural number $h$. 
\\
Furthermore, for any $x\in\mathfrak D^+_\infty$ we have
$$\varphi(xa)=\lim_{\alpha\rightarrow \infty} \varphi (\mathcal E_{n_\alpha}(x a))= \lim_{\alpha\rightarrow \infty} \varphi (x \mathcal E_{n_\alpha}(a)) = \varphi(x  \mathcal E_+(a))$$
it follows that we have a unique $\sigma$-limit point $\mathcal E_+ (a)$ for the net  $ \left\{ \mathcal E_{n} (a) \right\}_{n\in\mathbb N}$.
\\ 
Therefore we obtain a map $\mathcal E_+ :\mathfrak M \rightarrow  \mathfrak{D}_\infty ^+$.
\\
Moreover $\mathcal E_{n_\alpha}( \mathcal E_+(a)) = \mathcal E_+(a)$ for all $\alpha$, then $ \mathcal E_+ ^2 = \mathcal E_+ $ and for Tomiyama \cite{tomiyama57} the positive map $\mathcal E_+$ is a conditional expectation such that $\varphi \circ \mathcal E_+=\varphi$, precisely it is the conditional expectation of relation \ref{teo-dec1}.
\\
We can say something more:
\begin{proposition}
Let $\{ \mathfrak{M}, \Lambda_k, \varphi \}_{k\in\mathbb N}$ be a family of quantum dynamical systems. We consider the contraction   $V_k:\mathcal H_\varphi \rightarrow \mathcal H_\varphi$ defined in (\ref{contrazione 1}) related to Schwartz map $\Lambda_k $:
$$ V_k \pi_\varphi (a)\Omega_\varphi=\pi_\varphi(\Lambda_k (a))\Omega_\varphi \qquad a\in\mathfrak M$$
If $|| \left[ V_k^* - V_h ^* \right]\xi||\rightarrow 0 $ as $h,k\rightarrow \infty$ for all $\xi\in\mathcal H_\varphi$, then there is a unital positive map $\Lambda:\mathfrak M \rightarrow \mathfrak M$ such that 
\begin{equation}
\label{mohari1}
|| \phi \circ \Lambda_k - \phi \circ \Lambda ||\rightarrow 0
\end{equation}
as $k\rightarrow \infty$ for any $\phi\in\mathfrak M_*$ with
$$ \varphi(\Lambda(a^*) \Lambda(a))\leq\varphi(a^*a) \qquad  a\in\mathfrak M$$
and $\varphi \circ \Lambda = \varphi$.
\end{proposition}
\begin{proof}
A simple consequence of proposition 1.1 of \cite{mohari} 
\end{proof}
For each natural number $n$, we consider the following Schwartz map:
$$ Z_n = \frac{1}{2n+1}\sum \limits_{k=-n}^{n} \  \tau_k$$
it is obvious that $\varphi$ is a stationary state for $Z_n$ with $\varphi(x Z_n(y))=\varphi(Z_n (x)y)$ for all $x,y\in\mathfrak M$. 
\\
Moreover for each $a\in\mathfrak M$ we have:
\begin{eqnarray*}
 \pi_\varphi(Z_n(a)) \Omega_\varphi  & = & \frac{1}{2n+1}\sum \limits_{k=-n}^{n} \  \pi_\varphi(\tau_k(a))\Omega_\varphi =
\\
& = &
\frac{1}{2n+1}\sum \limits_{k=0}^{n} \  U_{\Phi,\varphi}^{* k} U_{\Phi,\varphi}^n  \pi_\varphi(a))\Omega_\varphi +
\frac{1}{2n+1}\sum \limits_{k=1}^{n} \  U_{\Phi,\varphi}^k U_{\Phi,\varphi}^{* k}  \pi_\varphi(a))\Omega_\varphi
\end{eqnarray*}
and since $U_{\Phi,\varphi}^{* n} U_{\Phi,\varphi}^n \rightarrow V_+$ and  $U_{\Phi,\varphi}^{n} U_{\Phi,\varphi}^{* n}\rightarrow V_-$ in strong operator topology, we obtain 
$$ \pi_\varphi(Z_n(a)) \Omega_\varphi \rightarrow \frac{1}{2}(V_+ + V_-)  \pi_\varphi(a)\Omega_\varphi  $$
It follows that from previous proposition that there is a $\varphi$ invariant Schwartz map $Z:\mathfrak M \rightarrow \mathfrak M$ such that 
\[
   ||\phi \circ Z_n-\phi \circ Z|| \rightarrow 0 \qquad \phi\in\mathfrak{M}_* 
\]
and 
$$ \pi_\varphi(Z(a)) \Omega_\varphi =  \frac{1}{2}(V_+ + V_-)  \pi_\varphi(a)\Omega_\varphi$$
We consider the decomposition $\mathfrak M = \mathfrak D_\infty \oplus \mathfrak D_\infty^{\perp_\varphi}$ for each $a=a_\| + a_\bot \in\mathfrak M$ result
$$Z(a_\|+a_\bot) = a_\|+Z(a_\bot)$$
with $Z(a_\bot)\in\mathfrak D_\infty^{\perp_\varphi}$.
\\
We observe that if  $\Phi^n(d_\bot)\rightarrow 0$   and $ \Phi^{\sharp n}(d_\bot)\rightarrow 0  $ \ as \ $n\rightarrow \infty$ in $s$-topology for all $d_\bot\in\mathfrak{D}_\infty^{\perp_\varphi}$ (see proposition \ref{s-conv}); then  $Z(d_\bot )=0$ for all $d_\bot\in\mathfrak{D}_\infty^{\perp_\varphi}$ and we have a  $\varphi$ invariant Schwartz map $Z:\mathfrak M \rightarrow \mathfrak D_\infty$ such that 
$$ Z(xa)=xZ(a), \qquad x\in\mathfrak M, \ a\in\mathfrak D_\infty $$
It follows that $Z$ is the conditional expectation $\mathcal E_\infty$ of proposition \ref{teo-dec2}.

\end{document}